\documentclass[12pt]{article}%
\usepackage [a4paper,left=1.8cm,right=1.8cm,top=2cm]{geometry}
\usepackage{amssymb}
\usepackage{amsfonts, color}
\usepackage{graphicx}
\usepackage{amsmath}%
\providecommand{\U}[1]{\protect\rule{.1in}{.1in}}
\newtheorem{theorem}{Theorem}[section]

\newtheorem{lemma}[theorem]{Lemma}

\newtheorem{proposition}[theorem]{Proposition}
\newtheorem{remark}[theorem]{Remark}

\newenvironment{proof}[1][Proof]{\textbf{#1.} }{\ \rule{0.5em}{0.5em}}

\usepackage{verbatim}

\newcommand{\diam}{{\rm diam}}
\newcommand{\cat}{{\rm cat}}

\newcommand{\Hnod}{{\rm \text{\scriptsize{H-}}nod}}
\newcommand{\Hf}{{\rm \text{\scriptsize{H}}}}

\newcommand{\Hscap}{{\rm \text{\scriptsize{H-}}cap}}
\newcommand{\meiocap}{\text{\scriptsize{$(1/2)$-}}{\rm cap}}

\newcommand{\Gr}{{\rm Gr}}

\newcommand{\HH}{\mathbb{H}^2}

\newcommand{\Hhorocoisa}{{\rm \text{\scriptsize{H-}}horosurface}}
\newcommand{\Hhoro}{\Hnod_{\infty}}

\title{Constant mean curvature surfaces in $\HH\times\mathbb{R}$ with  boundary in two parallel planes}
\author{Patricia Klaser, Rodrigo Soares, Miriam Telichevesky}
\begin{document}

\maketitle

%
\begin{abstract}
Given $H\in [0,\infty),$ some sufficient conditions for existence of CMC $H$ graphs with boundary in two parallel planes of $\mathbb{H}^2\times\mathbb{R}$ are presented. Height estimates for outwards-oriented CMC surfaces (horo)cyllindrically bounded are also exhibited.
\end{abstract}

\section{Introduction}

We study the existence and non existence of constant mean curvature (CMC) surfaces with boundary in two parallel planes in $\HH\times\mathbb{R},$ where $\mathbb{H}^2$ is the hyperbolic plane of curvature $-1.$ This issue was first treated in the Euclidean $3-$space, beginning with the classification of minimal annuli with boundary two circles in parallel planes by B. Riemann in 1898 \cite{RI}, which was concluded M. Shiffman \cite{SHI} in 1956. In \cite{AH} A. Ros and H. Rosenberg proposed the problem of finding a CMC annulus having as boundary two given convex Jordan curves in distinct parallel planes of $\mathbb{R}^3.$ In this setting, J. Ripoll and N. Espirito Santo \cite{NJ} considered the case of the boundary curves $\alpha$ and $\beta,$ in two parallel planes, are such that the orthogonal projection of $\alpha$ in the plane that contains $\beta$ is inside the region bounded by $\beta$. With hypothesis that relate the mean curvature $H$, the distance $h$ between the parallel planes and the curvatures of 
$\alpha$ and $\beta,$ an existence (Theorem 2.2) and a non existence (Proposition 3.2) result for a CMC surface with boundary $\beta \cup \alpha$ are proved. This problem was also studied by A. Aiolfi, P. Fusieger and J. Ripoll in \cite{AF}, \cite{AFR} and \cite{FR}, where  the authors considered radial graphs and demonstrated existence results.
 
If one replaces the ambient space $\mathbb{R}^3$ by a product space $M\times \mathbb{R},$ the problem described above still makes sense. One asks whether there are CMC surfaces with boundary two convex curves in $M\times \{0\}$ and $M\times \{h\}.$ L. Hauswirth \cite{Hau} studied generalized Riemann minimal surfaces in $\mathbb{H}^2\times \mathbb{R}$ and $\mathbb{S}^2\times \mathbb{R},$ by adapting the work of M. Shiffman \cite{SHI} and A. Barbosa \cite{AT} established a version of 
 results for minimal surfaces of \cite{NJ} in $\mathbb{H}^{2}\mathbb{\times R}$ by using Perron's Method. Other related results in $\mathbb{H}^{2}\times\mathbb{R}$ can be found in \cite{BRWE}, \cite{LHJ}, \cite{SET}, \cite{GC} and \cite{MBR}.

In this paper we treat the CMC case in the ambient space $\mathbb{H}^2\times \mathbb{R}.$ We prove existence and non existence results for CMC surfaces with boundary consisting in two Jordan curves $\beta \subset \mathbb{H}^2\times \{0\}$ and $\alpha \subset \mathbb{H}^2\times \{h\}$ and such that the orthogonal projection 
of $\alpha$ in $\mathbb{H}^2\times \{0\},$ which we also denote by $\alpha,$ is inside the region bounded by $\beta.$ We orient the CMC surface so that the normal vector $N$ points downwards, in the sense that it points to the negative part of the real axis,
 and we assume $H\ge0.$ Existence results for $h\ge 0,$ which corresponds to $\alpha$ above $\beta,$ and $h<0$ are then presented. It can also be observed that for $H>1/2$ the existence hypothesis are more restrictive, since 
the behaviour of some natural CMC barriers depends on which of the intervals $[0,1/2]$ or $(1/2,+\infty)$ the mean curvature $H$ belongs to. As far as the authors know, the case $H>1/2$ has not been treated in the literature yet.

In our setting, we look for a function $u \in C^{2}(\Omega)\cap C^{0}(\overline{\Omega})$ that has as graph a CMC $H$ surface with boundary $\alpha \cup \beta.$ This occurs if and only if $u$ is solution of the following Dirichlet problem%

\begin{equation}\label{eq_problemaPH}
\left\{
\begin{array}
[c]{l}%
Q_{H}(u)=\mathrm{div}\left(  \frac{\nabla u}{\sqrt{1+\rvert\nabla u\rvert^{2}%
}}\right)  +2H=0  \text{ in }\Omega\\
u\rvert_{\alpha}=h  \\
u\rvert_{\beta}=0
\end{array}
\right.
\end{equation}
where $\Omega\subset\mathbb{H}^{2}$ is the annular domain enclosed by
$\alpha \cup \beta$ in $\mathbb{H}^2\times\{0\}$ 
 and $\operatorname{div}$ and $\nabla$ are the divergence and the gradient in $\mathbb{H}^{2}$.

The proofs of the main results are based on construction of geometric barriers, 
that are CMC surfaces, which we name $\Hf-$nodoid, catenoid and $\Hf-$spherical cap. They are produced rotating graphs of certain functions, which are detailed in Section \ref{sec-rotacionais}. However, the statements of the main theorems require the knowledge of the $\Hf-$nodoid of inner radius $r.$ Its name is inspired by its behavior, which reminds the one of the nodoids in $\mathbb{R}^3$ and it is generated by the function

\begin{equation}\label{eq-Hnod}
\Hnod_r(s)=\displaystyle{\int_0^{s} \frac{2H(\cosh(r)-\cosh(r+t))+\sinh(r)}{\sqrt{\sinh^2(r+t)-(2H(\cosh(r)-\cosh(r+t))+\sinh(r))^2}}dt,}\end{equation}
 defined in $[0,T_H],$ where


\begin{equation}\label{eq-TH}T_H:=\left\{\begin{array}{l}+\infty,\text{ if } H\le 1/2,\\ \ln\left(\frac{2H+1}{2H-1}\right), \text{ if }H>1/2.\end{array}\right. \end{equation}

The fact that $T_H$ does not depend on $r$ is remarkable: For $H>1/2,$ the distance between two vertical slope points of $\Hnod$ does not depend on the size of its neck. If, in the limit, the neck is taken with radius 0, then $\Hnod_0=\Hscap$ corresponds to the $\Hf-$spherical cap and $2T_H$ is the distance between two vertical tangents of the cap. 
It is possible to obtain an expression using elementary functions for $\Hscap,$ and that is the reason why it does not appear explicitly on the hypothesis of the results. We also include the so-called catenoid as a particular case of nodoid defining $\cat_r:=\,{\rm \text{\scriptsize{$0$-}}nod_r}.$ 

\

Let us fix some notation to state the main results. Let $\alpha$ and $\beta$ be two simple closed curves in $\HH\times\{0\}$ such that $\alpha$ is contained in the interior of the region bounded by $\beta.$ Suppose that $\alpha$ satisfies the interior circle condition of radius $r$ and that $\beta$ satisfies the exterior circle condition of radius $R$ (See Section \ref{sec-proofs} for a precise definition). Denote by $\Omega$ the domain bounded by $\alpha \cup \beta$ and by $d$ the distance between $\alpha$ and $\beta.$

\begin{figure}[tbh]
\centering
\includegraphics[scale=0.4]{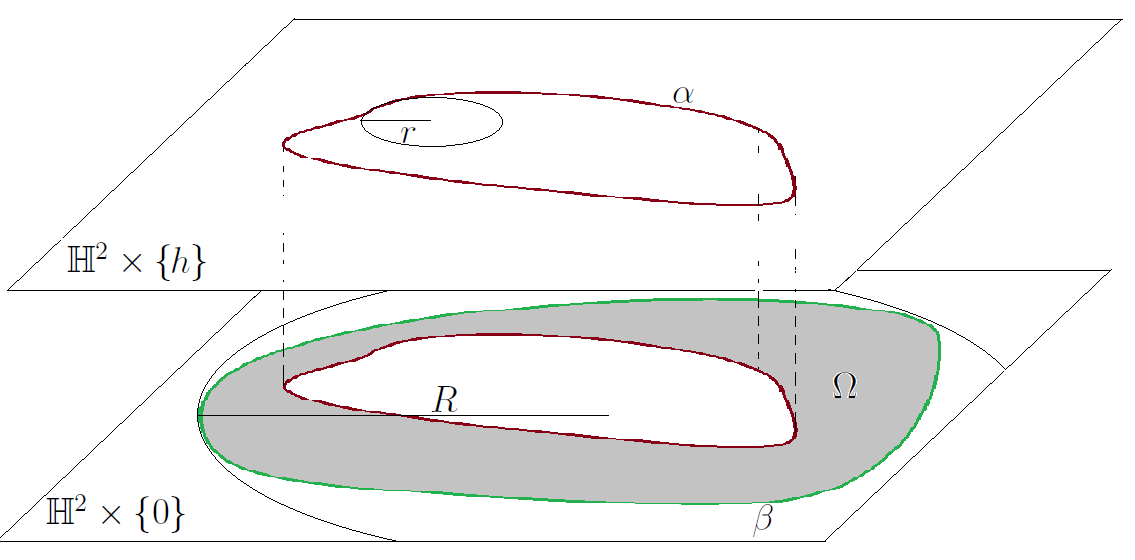}
\end{figure}


\begin{theorem}\label{teo-h>0}
Given $H\geq 0,$ consider Problem \eqref{eq_problemaPH} associated to $H.$ If 
\begin{equation}\label{eq-corollary}
\diam \beta \leq d+2 \cosh^{-1}\left(\cosh(r)+\frac{\sinh(r)}{2H}\right),
\end{equation}

\begin{equation}\label{eq-hipsubs}0\leq h\leq \max\left\{\cat_r(d),\frac{2Hd}{\sqrt{\coth^2(r)-4H^2}}\right\},
\end{equation}
which means $h<\infty$ if $\coth^2(r)-4H^2\leq 0;$
and additionally if $H>1/2,$ it holds that
\begin{equation}\label{eq-cap}
R\le T_H \text{ and }
0\leq h\leq 
\frac{4H}{\sqrt{4H^2-1}}\arctan\left(\sqrt{\frac{1-4H^2 \tanh^2\left(\frac{T_H-d}{2}\right)}{4H^2-1}}\right),
 \end{equation}
then Problem \eqref{eq_problemaPH} has a unique solution $u\in C^2(\Omega)\cap C^0(\overline{\Omega}).$
\end{theorem}

%
%
%
%

Given two Jordan curves $\alpha$ and $\beta$ in parallel planes of $\mathbb{H}^2\times \mathbb{R},$ there are two natural ways of searching CMC $H$ surfaces with boundary $\alpha\cup \beta.$ They consist in considering the two possible orientations of the surface. In the case of graphs as treated here, one corresponds to $h>0$ and the other to $h<0.$
For $h<0$, the mean curvature vector of the graph of the solution of \eqref{eq_problemaPH} points ``outwards'', in the sense that it points to the unbounded connected component of $(\HH\times[h,\max u])\setminus\Gr(u)$. In the context of orientation outwards we have an existence and non-existence results, as below.

\begin{theorem}\label{teo-h<0}
Let $H\geq 0,$ if $\diam \beta - (2r+d)\le T_H,$ $R\le T_H,$ and
\begin{equation}0 \leq -h\le \min\{\Hnod_r(d), \Hnod_r(\diam \beta - (2r+d))\}
\label{eq-teoh<0nod}
\end{equation}
then Problem \eqref{eq_problemaPH} has a unique solution $u\in C^2(\Omega)\cap C^0(\overline{\Omega}).$
\end{theorem}

\

The non-existence results are stated below. They assure that if the distance between the two parallel planes is large, then there is not a CMC surface with boundary in them and mean curvature vector pointing outwards. We point out that these results apply also to CMC surfaces that are not graphs and the first one holds even for unbounded surfaces. Both statement and proofs of Theorems \ref{theo-nonexist} and \ref{teo-nonexistBr} are adapted from the situation on $\mathbb{R}^3$ presented in \cite{NJ}. We emphasize that the estimates are sharp in the sense that the $\Hhorocoisa$s and the $\Hf-$nodoids (see Section \ref{sec-rotacionais}) provide examples.

\begin{theorem}\label{theo-nonexist}
Let $M\subset \HH\times \mathbb{R}$ be an immersed connected $H-$surface contained in an horocylinder $B\times \mathbb{R}$, where $B$ is an horodisc. Suppose that $M$ intersects two parallel slices $\HH\times \{0\}$ and $\HH\times\{h\}$, $\partial M \cap (\HH\times (0,h))=\emptyset$
and that the mean curvature vector of $M$ points to $\partial B\times \mathbb{R}.$ 

Then $h\le h_H$, where $h_H$ is given by $$ h_H:=\left\{\begin{array}{l}\pi -\frac{8H}{\sqrt{1-4H^2}}\tanh^{-1}\left(\frac{1-2H}{\sqrt{1-4H^2}}\right),\text{ if }H< 1/2,\\ \pi-2,\text{ if }H=1/2, \\ \pi - \frac{8H}{\sqrt{4H^2-1}} \tan^{-1}\left(\frac{2H-1}{\sqrt{4H^2-1}}\right),\text{ if }H>1/2.\end{array}\right.$$
\end{theorem}

If $M\subset \HH\times \mathbb{R}$ is an immersed connected $H-$surface contained in an ordinary cylinder $B_{r^*}\times \mathbb{R},$ the conclusion is actually that $h\le 2\max \Hnod_{r^*} \le h_H.$

\begin{theorem}\label{teo-nonexistBr}
Let $M\subset \HH\times \mathbb{R}$ be an immersed connected $H-$surface contained in a cylinder $B_{r^*}\times \mathbb{R}$, where $B_{r^*}$ is a disk of radius $r^*.$ Suppose that $M$ intersects two parallel slices $\HH\times \{0\}$ and $\HH\times\{h\}$, $\partial M \cap (\HH\times (0,h))=\emptyset$
and that the mean curvature vector of $M$ points to $\partial B_{r^*}\times \mathbb{R}.$  Then $$h\le 2\max \Hnod_{r^*}.$$\end{theorem}

Although we do not have an explicit expression for $\max \Hnod_r,$  the behavior of the family $\Hnod_r$ with $r$ implies that it is at most $h_H/2,$ which makes Theorem \ref{teo-nonexistBr} a particular case of Theorem \ref{theo-nonexist}. However, since $\max \Hnod_r$ converges monotonically to $0$ as $r$ goes $0,$ Theorem \ref{teo-nonexistBr} is a stronger result if $M$ is compact.

\section{Rotational surfaces}\label{sec-rotacionais}

The proofs of our main results require some barriers in order to see that the Perron solutions are continuous up to the boundary. In most cases, we take as barriers rotational surfaces of $\HH\times\mathbb{R},$ studied first in \cite{STM} and also in \cite{BRWE} and \cite{STP}.

Consider $a\in \mathbb{H}^2$ and denote by $C_a(r)$ (respectively $B_a(r)$) the geodesic circle (respectively disk) centered at $a$ with radius $r.$ Denote by $s:\mathbb{H}^2\setminus B_a(r) \to [0,+\infty)$ the distance function to $C_a(r),$ defined outside $B_r(a).$ If $u:[0,T] \to \mathbb{R}$  is a real function, the rotational surface generated by $u$ around the axis $\{a\}\times \mathbb{R}$ is defined as the set $\{(p,u(s(p))\,|\,p\in \HH\setminus B_r(a)\}.$ Observe that the Laplacian of $s$ is given by $\Delta s = \coth (s+r),$ so in order that the  rotational surface to be an $H-$surface in $\HH\times \mathbb{R},$ $u$ must satisfy 
\begin{equation}\label{eq-solution}
\left(\frac{u'(s)}{\sqrt{1+u'^2(s)}}\right)'+\left(\frac{u'(s)}{\sqrt{1+u'^2(s)}}\right)\coth (s+r) +2H=0.
\end{equation}

%

Integrating this ODE with initial condition $u'(0)=+\infty$ we obtain the general expression corresponding to fixed $r$ and $H$, which we call $\Hf-$nodoid of CMC $H$ and inner radius $r:$

\begin{equation*}
\Hnod_r(s)=\displaystyle\int_0^s \frac{2H(\cosh(r)-\cosh(t+r))+\sinh(r)}{\sqrt{\sinh^2(t+r)-(2H(\cosh(r)-\cosh(t+r))+\sinh(r))^2}}dt,
\end{equation*}
defined in $[0,T_H],$ where $T_H$ is defined in \eqref{eq-TH}.

The $\Hf-$nodoids form a $2-$parameter continuous family of CMC surfaces in $\HH\times \mathbb{R}.$ The parameter $H\in [0,+\infty)$ gives the mean curvature of the surface and the parameter $r\in [0,+\infty]$ gives the radius of their necks ($r=0$ means no neck and $r=+\infty$ means that the neck is a horocycle). Since some particular cases in this family are important to our constructions, we begin by describing them. 

\begin{center}
\begin{figure}[tbh]
\centering
\includegraphics[scale=0.4]{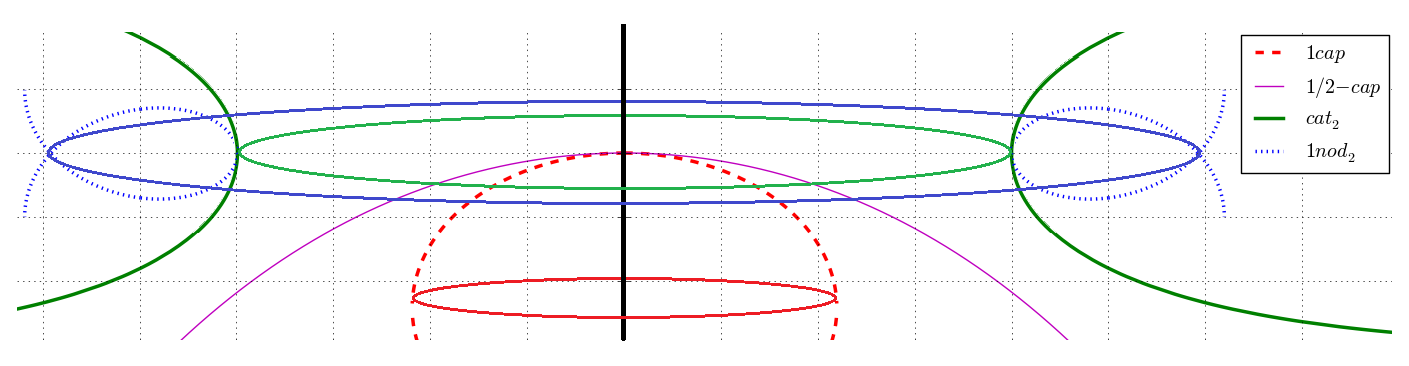}
\end{figure}
\end{center}

\subsection*{The caps for $H\geq 1/2$} 
For $H\ge 1/2,$ taking the limit $r\to 0^+$ of $\Hnod_r$ we get the $\Hf-$spherical cap. Expression \eqref{eq-Hnod} is then integrable.

If $H>1/2,$
\begin{equation}
\begin{array}{ll}\label{eq-def-cap}
\Hscap(s)&=2H\displaystyle\int_0^s \frac{(1-\cosh(t))}{\sqrt{\sinh^2(t)-4H^2(1-\cosh(t))^2}}dt \\
&\displaystyle{=\frac{4H}{\sqrt{4H^2 -1}}\left.\left[\arctan\left(\sqrt{\frac{1 - 4H^2\tanh^2(\frac{s}{2})}{4H^2 -1}}\right)\right]\right|^{s}_{0}    }
\end{array}
\end{equation}

We observe that $$\max|\Hscap|=-\Hscap(T_H)= \frac{4H}{\sqrt{4H^2 -1}}\arctan\left(\sqrt{\frac{1}{4H^2 -1}}\right).$$

For $H=1/2,$
\begin{equation}
\meiocap(s)=2-2\cosh(s/2)
\label{eq:meiocap}
\end{equation}
and $\lim_{s\to \infty} \meiocap(s)=-\infty,$ therefore, we may keep in mind that $\meiocap$ has infinite height.

\subsection*{The catenoids $\cat_r$}
The catenoid is the minimal $\Hnod_r,$ that is, it is the ${\rm \text{\scriptsize{$0$-}}nod_r}.$ It is the only $\Hnod$ given by an strictly increasing function, which is asymptotic to the horizontal plane $\mathbb{H}^2\times \{\cat_r(\infty)\}.$


\begin{equation}\label{eq-def-cat}
\displaystyle\cat_r(s)=\displaystyle\int_0^s \frac{\sinh(r)}{\sqrt{\sinh^2(t+r)-\sinh^2(r)}}dt, \, s\in (r,+\infty).
\end{equation}

It turns out that $\cat_r \le \pi/2,$ as it is proved in \cite{BRWE}. This is also a consequence that $r\mapsto \cat_r(s)$ is increasing and that the supremum of the horocatenoid $\cat_{\infty}$ is $\pi/2,$ see below.

\subsection*{Horonodoids and horocatenoids}\label{subs-horo}

Letting $r\to +\infty$ in expression \eqref{eq-Hnod}, we obtain the $\Hf-$horonodoids. 

The $\Hf-$horonodoids also may be seen as the $\Hf-$surfaces obtained with rotations around a point at infinity. These $\Hf-$surfaces are graphs of functions whose level sets are horocycles with the same point at infinity. To find them we solve $Q_H(u)=0,$ but in this case we assume that $u=u(s)$ depends on the distance $s$ to a horocycle $C\in \HH.$ If $C$ is the boundary of the horodisk $B,$ then in $\HH\setminus \overline{B},$ it holds $\Delta s=1$ (notice that it is also a limit case: $\coth(r+s)$, the Laplacian of the distance of a circle of radius $r,$ tends to $1$ as $r$ tends to $\infty$). Therefore in order to obtain a CMC $\Hf-$surface, $u$ must satisfy
$$\left(\frac{u'(s)}{\sqrt{1+u'^2(s)}}\right)'+\left(\frac{u'(s)}{\sqrt{1+u'^2(s)}}\right) +2H=0.$$ We solve this ODE for the initial condition $u'(0)=+\infty,$ which means that the neck of the surface is the horocycle $C,$ and we find that the function is (the same as obtained letting $r\to \infty$ in the definition of $\Hnod_r$)

\begin{equation}\label{eq-horocoisa}
\begin{array}{ll}
\Hhoro(s)&\displaystyle{=\displaystyle\int_0^{s}\frac{-2H+(1+2H)e^{-t}}{\sqrt{1-(-2H+(1+2H)e^{-t})^2}}dt} \vspace{0.2cm}\\
&\displaystyle{=\left\{\begin{array}{l}\theta - \frac{4H}{\sqrt{1-4H^2}}\tanh^{-1} \left(\sqrt{\frac{1-2H}{1+2H}}\tan\left(\frac{\theta}{2}\right)\right),\text{ if }H<1/2, \vspace{0.2cm} \\ \vspace{0.2cm} \theta - \tan\left(\frac{\theta}{2}\right), \text{ if }H=1/2, \\ \vspace{0.2cm} \theta -\frac{4H}{\sqrt{4H^2-1}}\tan^{-1} \left(\sqrt{\frac{2H-1}{2H+1}}\tan\left(\frac{\theta}{2}\right)\right), \text{ if }H>1/2,
\end{array}\right.}
\end{array}
\end{equation}
where $0\le\theta\le \pi$ is such that $\cos \theta = (1+2H)e^{-s} - 2H.$

\begin{remark}If we use the complex identity $\tan(iz)=i\tanh(z)$ 
we may summarize all expressions in \eqref{eq-horocoisa} by
\begin{equation*}
\Hhoro(s)\displaystyle=\theta - \frac{4H}{\sqrt{1-4H^2}}\tanh^{-1} \left(\sqrt{\frac{1-2H}{1+2H}}\tan\left(\frac{\theta}{2}\right)\right).
\end{equation*} The same may be done for the maximum value of $\Hnod_{\infty},$ given in \eqref{eq-maxHhoro}.\end{remark}

The domain of $\Hnod_{\infty}$ is the same as the $\Hnod_r$ and the maximum height of $\Hnod_{\infty}$ is attained at $x_H(\infty) = \ln(1+1/2H),$ its value is   

\begin{equation}\label{eq-maxHhoro}
\begin{array}{ll}
\Hhoro(\ln\left(\frac{2H+1}{2H}\right))=\left\{\begin{array}{l}\frac{\pi}{2} -\frac{4H}{\sqrt{1-4H^2}}\tanh^{-1}\left(\frac{1-2H}{\sqrt{1-4H^2}}\right)\text{ if }H< 1/2,\vspace{0.2cm} \\ \frac{\pi}{2}-1,\text{ if }H=1/2,\vspace{0.2cm}\\ \frac{\pi}{2} -\frac{4H}{\sqrt{4H^2-1}}\tan^{-1}\left(\frac{2H-1}{\sqrt{4H^2-1}}\right)\text{ if }H> 1/2.\end{array}\right.
\end{array}
\end{equation}

The $\Hhorocoisa$ of neck $C$ is the surface obtained gluing the graph of $\Hhoro$ and the graph of $-\Hhoro.$ 
It is clear that for any horocycle $C\subset\HH$ and any $H\geq 0,$ there is an associated $\Hhorocoisa$ of neck $C.$ Besides, if $C'$ is a horocycle equidistant from $C,$ then the $\Hhorocoisa$ associated to $C'$ is isometric to the one associated to $C.$

Finally, notice that the horocatenoid is a particular case of $\Hf-$horonodoid: it is the $0-$horonodoid. Therefore, it may be obtained either letting $r\to +\infty$ for $\cat_r$ or taking $H=0$ for $\Hnod_{\infty}.$ It has the explicit expression $$\cat_{\infty}(s)=\int_0^{s}\frac{e^{-t}}{\sqrt{1-e^{-2t}}}dt=\frac{\pi}{2}-\sin^{-1}(e^{-s}).$$

\subsection{More about $H-$nodoids}\label{sec-moreHnod}
Recall that $\Hf-$nodoid of neck of radius $r$ is obtained rotating the graph of
\begin{equation*}
\Hnod_r(s)=\displaystyle\int_0^s \frac{2H(\cosh(r)-\cosh(t+r))+\sinh(r)}{\sqrt{\sinh^2(t+r)-(2H(\cosh(r)-\cosh(t+r))+\sinh(r))^2}}dt.
\end{equation*}

The rotational surface may then be reflected about the plane $\HH\times\{0\}$ and glued with its reflection to build a CMC surface. It turns out that this surface, if $H>0,$ intersects itself on two circles, of radius $r$ and $\rho_H(r)$, centered at $a,$ what justifies the name nodoid. For $H>1/2,$ the domain of $\Hnod_r$ is $[0, T_H]$ and after the value $T_H$, it is possible to continue this CMC rotational surface corresponding to $\Hnod_r$, but it is not a vertical graph over $\mathbb{H}^2\times\{0\}$ anymore. 

\begin{center}
\begin{figure}[tbh]
\centering
\includegraphics[scale=0.7]{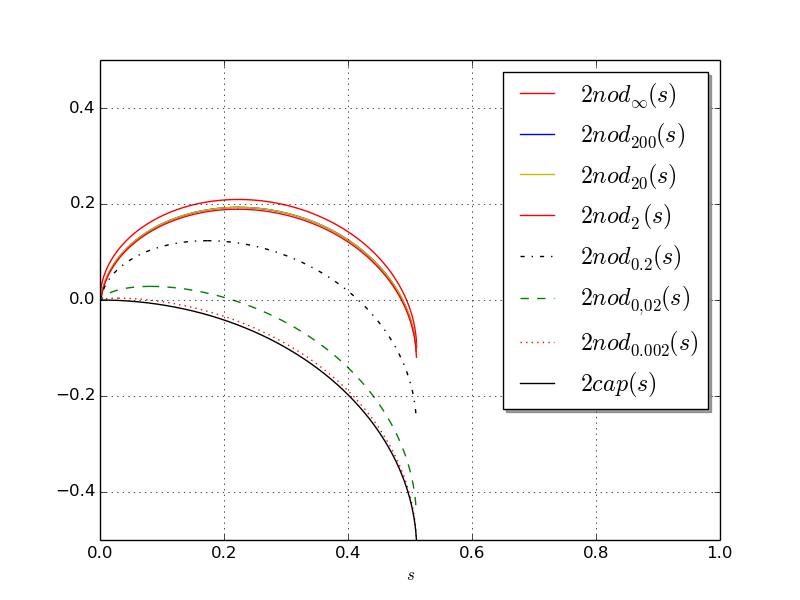}
\end{figure}
\end{center}

Let us list some important facts about $\Hnod_r$, that are used throughout this article.

\begin{description}
	\item[Fact 1.] For fixed $H,$ and $s\in [0,T_H],$ the function that assigns for each $r>0$ the value $\Hnod_r(s)$ is strictly increasing.
	 
	Indeed, a computation gives $$\frac{\partial \Hnod_r(s)}{\partial r}>0.$$
	
	In particular:
	
	\begin{description}
		\item[a)] if $0<H\le 1/2,$ since $\lim_{s\to + \infty} \Hnod_{\infty}(s) = -\infty,$ we obtain that $$\lim_{s\to +\infty} \Hnod_r(s) = -\infty$$ for every $r> 0.$ 
		\item[b)] if $H>1/2,$ we have estimatives for $\Hnod_r(T_H)$ from above and from below:
	$$-\frac{4H}{\sqrt{4H^2-1}} \tan^{-1}\left(\frac{1}{\sqrt{4H^2-1}}\right) \le \Hnod_r(T_H) \le \pi\left(1-\frac{2H}{\sqrt{4H^2-1}}\right).$$
	\end{description}	
	
	Moreover, fixed $H,$ the maximum value of $\Hnod_r$ is bounded from above by $\max \Hnod_{\infty},$ that is, for all $s\in [0,T_H],$ $$\Hnod_r(s) \le \frac{\pi}{2} - \frac{4H}{\sqrt{4H^2-1}} \tan^{-1}\left(\frac{2H-1}{\sqrt{4H^2-1}}\right).$$
	
	\item[Fact 2.] The maximum value of $\Hnod_r$ is attained at $$x_H(r):=\cosh^{-1}\left(\cosh r + \frac{\sinh r}{2H}\right) - r.$$ It is easy to check that $x_H(r)$ is increasing with $r,$ and therefore $$0\le x_H(r) \le \ln \left(1+\frac{1}{2H}\right).$$ 
	
We also remark that for any $r>0,$ $x_0(r) = +\infty,$ which is a natural conclusion: the maximum value of $\cat_r$ is attained at infinite distance, since it is strictly increasing and has horizontal asymptote.
	
	\item[Fact 3.] Let $\rho_H(r)$ denote the positive zero of $\Hnod_r,$ that always exists for $H>0$ and $r>0.$ The positivity of $\Hnod_r$ in $(0,\rho_H(r))$ is important to $\Hnod_r$ be an upper barrier. Although we do not have an explicit expression for $\rho_H(r),$ the next proposition provides an estimate for it from below, which provides hypothesis \eqref{eq-corollary} on Theorem \ref{teo-h>0}. Its proof is inspired in Proposition 1 of \cite{Jaime}, that corresponds to an analogous result in $\mathbb{R}^3.$
\end{description}

\begin{proposition}\label{miriam}
Given $r>0$ and $H>0,$ we define $x_H(r)>0$ as the positive solution of
\begin{equation}\label{eq-x0}
2H\cosh(x_H(r)+r)= 2H\cosh(r)+\sinh(r).
\end{equation}
Then $\Hnod_r(s)\ge 0$ for any $s\in [0,2x_H(r)].$
\end{proposition}

\begin{proof}
We first notice that $\Hnod_r(0)=0$ and $x_H(r)$ is the only absolute maximum of $\Hnod_r$. Therefore, it suffices to prove that $\Hnod_r$ decreases in $[x_H(r),2x_H(r)]$ slower than it increases in $[0,x_H(r)]$. For that, we verify the inequality $$|\Hnod_r'(x_H(r)+t)|\le |\Hnod_r'(x_H(r)-t)|\text{ for }t\in (0,x_H(r)).$$

To simplify notation, let us denote $x_H(r)+r$ by $x.$ Notice that $\Hnod_r'(s) = \frac{g(r+s)}{\sqrt{1-g(r+s)^2}},$ where $$g(t)=\frac{2H\cosh r + \sinh r - 2H \cosh t}{\sinh t}.$$ It implies that we must prove that $$-\frac{g(x+t)}{\sqrt{1-g(x+t)^2}}\le \frac{g(x-t)}{\sqrt{1-g(x-t)^2}},$$ which is equivalent to $-g(x+t) \le g(x-t)$.The definition of $x_H(r)$ implies that the previous inequality is equivalent to
$$\frac{-\cosh(x)+\cosh(x+t)}{\sinh(x+t)}\le \frac{\cosh(x)-\cosh(x-t)}{\sinh(x-t)},$$ 
and a straightforward computation gives that it is equivalent to $\sinh(x) (\cosh(t) -1) \ge 0,$ which is of course valid for any $t\in \mathbb{R}.$
\end{proof}


\begin{description}
\item[Fact 4.] For fixed $r,$ the function $H\mapsto \Hnod_r(s)$ is decreasing (with ``decreasing domain'' as well). More precisely, if 
${\rm \text{\scriptsize{H$_1$-}}nod}(s)$ and ${\rm \text{\scriptsize{H$_2$-}}nod}(s)$ are defined for $H_1<H_2,$ then ${\rm \text{\scriptsize{H$_1$-}}nod}(s)>{\rm \text{\scriptsize{H$_2$-}}nod}(s).$
\end{description}


\section{Proofs of the existence results}\label{sec-proofs}

The setting of our existence results is the following:
Let $\Omega$ be the domain enclosed by $\alpha$ and $\beta,$ two simple closed curves in $\HH\times\{0\}$ such that $\alpha$ is contained in 
the interior of the region bounded by $\beta.$ Let $r>0$ such that $\alpha$ satisfies the interior circle condition of radius $r,$ i.e., for any $p\in \alpha,$ there is a circle of radius $r$ tangent to $\alpha$ at $p$ and contained in the region bounded by $\alpha.$
Let $R>0$ be such that $\beta$ satisfies the exterior circle condition of radius $R,$ that is, for any $p\in \beta,$ there is a circle of radius $R$ tangent to $\beta$ at $p$ that contains the region bounded by $\beta.$ Finally, let $d$ denote the distance between $\alpha$ and $\beta.$ 

\

To prove the existence results we use Perron's Method, see Section 2.8 of \cite{GT}. Recall that a function $v\in C^0(\Omega)$ is called a subsolution (resp. $w$ is a supersolution) of $Q_H=0$ if for any $U\subset \subset \Omega$ and for any $u:U\to \mathbb{R}$ solution of 
$Q_H=0$ in $U$ with $v\vert_{\partial U}\leq u\vert_{\partial U}$ (resp. $w\vert_{\partial U}\geq u\vert_{\partial U}$), it holds that $v\leq u$ (resp. $w\geq u$) in $U.$ If $v\in C^2(\Omega),$ then $Q_H(v)\geq 0$ in $\Omega$ is equivalent to $v$ being a subsolution. Analogously for supersolutions with $Q_H\leq 0.$ The Comparison Principle holds for $Q_H$: if $v,w$ are respectively sub and supersolutions and $v\le w$ on $\partial \Omega$, then $v\le w$ in $\Omega$. 

To simplify notation, from now on we denote by $\varphi$ the function $\varphi: \partial \Omega\to \mathbb{R}$ given by $\varphi(p)=0$ if $p\in \beta$ and $\varphi(p)=h$ for $p\in \alpha.$ Hence Problem \eqref{eq_problemaPH} is equivalent to

\begin{equation}\label{eq_problemaPHlinha}
\left\{
\begin{array}
[c]{l}%
Q_{H}(u)=\mathrm{div}\left(  \frac{\nabla u}{\sqrt{1+\rvert\nabla u\rvert^{2}%
}}\right)  +2H=0  \text{ in }\Omega\\
u\rvert_{\partial\Omega}=\varphi.  
\end{array}
\right.
\end{equation}

Consider now the set
$$S = \{ z \in C^{2}\left(\Omega\right) \cap C^{0}\left(\overline{\Omega} \right) \,\vert \, z \text{ is a subsolution of }Q_{H} = 0, 
z|_{\partial \Omega} \leq \varphi\}.$$

The set $S$ is not empty since the constant function $v:=\min \varphi=\min\{0,h\}$ belongs to it. The following lemma proves that any element in $S$ is bounded from above.

\begin{lemma}\label{lema-supersol}
Let $\Omega\subset\HH\times\{0\}$ be the domain bounded by $\alpha \cup \beta,$ as above, and $H\geq 0$ a given constant. If $H > \frac{1}{2}$, suppose additionally that $$R\le T_H \text{ and }
h\leq \Hscap(T_H-d)-\Hscap(T_H).$$ 
Then Problem \eqref{eq_problemaPH} admits a global 
supersolution, namely, there is a function $w:\overline{\Omega}\rightarrow \mathbb{R},$ such that $Q_H(w)\leq 0$ in $\Omega$ and 
$w\vert_{\alpha}\geq h$ and $w\vert_{\beta}\geq 0.$
\end{lemma}

\begin{proof}

The idea of this proof is to find a spherical cap that wraps $\alpha \cup \beta,$ which is possible for $h\leq \Hscap(T_H-d)-\Hscap(T_H),$ if $H>1/2,$ and for any $h,$ if $H\le 1/2.$ 

For $0\leq H \leq \frac{1}{2}$, consider $L > R$ great enough such that 
\begin{equation}\label{eq-01}
2\cosh\left(\frac{L}{2}\right) \geq 2\cosh \left(\frac{L-d}{2}\right) + h.
\end{equation}
Let $B_a=B_a(L)$ be a disk of radius $L$ and center $a$ containing $\Omega$ and such that $C_a=\partial B_a$ is tangent to $\beta.$
Let $s$ be the distance to $a$ in $\mathbb{H}^2$ and define $w: \overline{\Omega}\rightarrow \mathbb{R}$
by 
\begin{equation}\label{eq-11}
w(q) = \meiocap(s(q)) - \meiocap(L). 
\end{equation}  
Since $\Omega \subset B_{a},$ $s\leq L$ in $\beta.$ Then, if $q\in \beta,$
\begin{equation}
 w(q) = \meiocap(s(q)) - \meiocap(L) \geq 0,
\end{equation}
because $\meiocap$ is a decreasing function. Therefore, $w_{\vert \beta} \geq 0.$ We claim that 
$w_{\vert\alpha} \geq h.$ Indeed, given $q \in \alpha$ we have $s(q) \leq L - d$ since 
the disk of center $a$ and radius $L - d$ contains $\alpha,$ hence
\begin{equation}\label{eq-13}
w(q) = \meiocap(s_{a}(q)) - \meiocap(L) \geq \meiocap(L - d) - \meiocap(L) \geq h.
\end{equation}
By the previous section, $Q_{H}(w) = -2(\frac{1}{2}) + 2H \leq 0$
in $\Omega$ and from the Comparison Principle, $z \leq w$ for all 
$z$ solution of \eqref{eq_problemaPH}.

In the case $H > \frac{1}{2}$, consider $C_{a}=C_a(T_H)$, a circle of radius $T_H,$ with $a$ and $s$ as above.
Define $w: \overline{\Omega} \rightarrow \mathbb{R}$ by 
\begin{equation}\label{eq-21}
 w(q) = \Hscap(s(q)) - \Hscap(T_H)
\end{equation}
so that, by Section \ref{sec-rotacionais}, the graph of $w$ is part of 
a spherical cap with CMC $H$ which implies $Q_H(w)= 0.$ Proceeding as above for $T_H$ replacing $L,$ and using \eqref{eq-cap}, we conclude that $w|_{\partial\Omega}\geq \varphi.$ Therefore, $w$ is an upper bound for $S.$
\end{proof}

\

It follows that the function $u:\Omega \to \mathbb{R}$ given by
$$u(x):= \sup\{z(x)\,|\, z\in S\}$$ is well defined. Standard arguments on Perron's Method also give that $u\in C^2(\Omega)$ and that $Q_H(u)=0$ in $\Omega.$

For proving that $u$ satisfies the boundary conditions $u_{\vert\beta} = 0$ and $u_{\vert\alpha} = h$, we construct barriers at any point of $\partial\Omega$. Lower and upper barriers at $p\in \partial \Omega$ are respectively sub and supersolutions $v,w\in C^2(\Omega)\cap C(\overline{\Omega})$ with $v(p)=w(p)=\varphi(p)$ and $v|_{\partial \Omega}\le\varphi$, $w|_{\partial \Omega}\ge \varphi$. After that, the uniqueness of the solution follows from the Comparison Principle.

Since the cases $h>0$ and $h<0$ are quite different, we now split the proofs, as we have done with the statements.

\

%

\begin{proof}[Proof of Theorem \ref{teo-h>0}]

{\bf Barriers at points of $\beta.$}
Recall that $h\ge 0$, and therefore $\alpha$ is above $\beta$. It follows immediately that the constant function $v\equiv 0$ is a lower barrier at any point of $\beta$. Furthermore, the supersolutions given by Lemma \ref{lema-supersol} are upper barriers for a point $p\in \beta$, if we choose the exterior circle tangent to $\beta$ at $p.$ 

\

{\bf Barriers at points of $\alpha.$}
For $p \in \alpha$, using the interior circle condition of radius $r$ in $\alpha$, we may consider a circle $C_{a}$ of radius $r$ and center $a=a(p)$, tangent to  $\alpha$ in $p$ and contained in the bounded region whose boundary is $\alpha.$ Let $s:\overline{\Omega}\to \mathbb{R}$ be the distance to $C_a.$

The subsolution at $p$ is either a catenoid of neck $r$ or a truncated cone of neck $C_a.$ The fact that one of them is below $\beta$ is a consequence of hypothesis \eqref{eq-hipsubs}, as explained below.

If $h\leq \cat_r(d),$ let $v_{p}:\overline{\Omega}  \rightarrow \mathbb{R}$ is given by $v_{p}(q) = -\cat_{r}(s(q)) + h$, $q \in \overline{\Omega}.$ Using the previous section we have $Q_{H}(v_{p}) = 2H \geq 0.$ Notice that if $q \in \beta$, then $s(q) \geq  d,$ which implies that $-\cat_{r}(s(q)) \leq -\cat_{r} (d)$ and therefore
$$v_{p}(q) = -\cat_{r}(s(q)) + h \leq -\cat_{r}(d) + h \leq 0.$$ 
Hence $(v_{p})_{\vert\beta} \leq 0.$ If $q \in \alpha$, since $C_{a}$ is contained in the bounded region of boundary $\alpha,$ we have $-\cat_{r}(s(q)) \leq 0.$ Hence $$v_{p}(q) = -\cat_{r}(s(q)) + h \leq h,$$ and we obtain $(v_{p})_{\vert\alpha} \leq h.$ Since $p \in C_{a}$ we have $s(p) = 0$ and $v_{p}(p) = h = \varphi(p).$ 

If $h> \cat_r(d),$ we take as a barrier a truncated cone of inner radius $r$ containing $C_a$ and such that at height $0,$ its intersection with the slice $\mathbb{H}^2\times\{0\}$ is inside $\Omega.$ Let $v_{p}:\overline{\Omega}  \rightarrow \mathbb{R}$ be given by $$v_{p}(q) = -\frac{h}{d}(s(q)) + h,$$ which satisfies
$$Q_{H}(v_{p})=\frac{-h/d}{\sqrt{1+(h/d)^2}}\coth(s+r)+2H.$$ Therefore, it is a subsolution in $\Omega$ if and only if
$$\frac{h}{\sqrt{d^2+h^2}}\leq \frac{2H}{\coth r},$$
that is a consequence of \eqref{eq-hipsubs}. It is clear that $v_p(p)=0$ and $v_p\leq \varphi$ on $\partial\Omega.$

\

To finish, the supersolution at $p\in \alpha$ is the $\Hf-$nodoid of neck $C_a.$ Indeed, consider the function $w_{p}: \overline{\Omega}  \rightarrow \mathbb{R}$ defined by $w_{p}(q) = \Hnod_{r}(s(q)) + h$ with $r$ and $a$ as defined before. By Section \ref{sec-rotacionais}, we have that $Q_{H}(w_{p}) = 0.$ Let $q \in \overline{\Omega}$ and $q' \in \beta$ the intersection of $\beta$ and the geodesic that starts at $q$ and passes through $a.$ Notice that 
$$d(q, q') = d(a, q) + d(a, q')= s(q)+s(q')+2r\geq s(q)+d+2r.$$ Using hypothesis \eqref{eq-corollary},
$$s(q)\leq \diam \beta - d-2r \leq 2x_H(r).$$ 
By Proposition \ref{miriam}, it follows that $\Hnod_{r}(s(q)) \geq 0$ and therefore $(w_{p})_{\vert \overline{\Omega}} \geq h.$ 

Moreover, 
$$w_{p}(p) = \Hnod_{r}(0) + h = h = \varphi(p).$$ We conclude by noticing that since $\alpha$ is between $C_a$ and $\beta,$ it follows immediately that $\Hnod_r(s(q))\ge 0$ for $q \in \alpha,$ what implies that $w_p \ge h$ on $\alpha.$
This completes the proof.  
\end{proof}

\

\begin{proof}[Proof of Theorem \ref{teo-h<0}]

Recall that in this case $h\le 0$, hence $\alpha$ is below $\beta.$

\

{\bf Barriers at points of $\beta.$}
Since $\beta$ is above $\alpha,$ the supersolutions given by Lemma \ref{lema-supersol} are upper barriers for a point $p\in \beta,$ if we choose the exterior circle tangent to $\beta$ at $p.$ 

To find the subsolutions associated to points in $\beta,$ consider the upside down problem for $H=0.$ It consists in finding the solution $z\in C^2(\Omega)\cap C^0(\overline{\Omega})$ of $Q_0=0$ that takes boundary value $-\varphi.$ By hypothesis $-h\leq \Hnod_r(d)$ holds, and Fact 4 (Section \ref{sec-moreHnod}) implies that $\Hnod_r(d)\leq \cat_r(d),$ therefore Theorem \ref{teo-h>0} applies for boundary data $-\varphi.$
Hence $v=-z$ has minimal graph and works as a barrier from below at all points of $\beta.$

\

{\bf Barriers at points of $\alpha.$}   
Since $\alpha$ is below $\beta$ and constant functions are subsolutions, $v\equiv -h$ is a lower barrier for all points of $\alpha.$
For the supersolutions, given $p\in \alpha,$ consider $C_a$ the circle of radius $r$ internally tangent to $\alpha$ at $p.$ Then take $w(q)=\Hnod_r(s(q))+h$ the nodoid of inner radius $r,$ depending on the distance $s$ to $C_a.$ We claim that $w\geq\varphi$ on $\beta.$ To see that, observe that for any $q\in \beta, d\leq s(q)\leq \diam \beta -(2r+d)$ and since $s\mapsto \Hnod_r(s)$ is a concave function, it attains its minimum value on the boundary. Therefore, for any $q\in \beta,$ $w(q)\geq \min\{w(d), w(\diam \beta-(2r+d))\}\geq 0$ from \eqref{eq-teoh<0nod}. We conclude that $w\geq\varphi$ on $\beta$ and that $w\geq h$ on $\alpha,$ so that it is an upper barrier.
\end{proof}

\

We briefly remark some aspects of the hypotheses of Theorem \ref{teo-h>0}:
\begin{enumerate}

\item The first inequality \eqref{eq-corollary} states that if $\beta$ is too far from $\alpha$ the Theorem does not apply. This agrees with the fact that if $\Omega$ contains a large annulus, then it only admits  solution for footnotesize $h.$ (see \cite{Senni}). 

\item Since $2(r+d)<\diam \beta,$ inequality \eqref{eq-corollary} also implies that $$\cosh(r+d/2)\le \cosh(r)+\frac{\sinh(r)}{2H}$$ which can be seen as an upper bound for $d.$

\item For $H\le 1/2,$ one of the existence hypothesis is \eqref{eq-hipsubs}. It is vacuous for $\coth(r) \le 2H,$ what happens because any truncated cone of inner radius $r,$ top inside $\alpha\subset \mathbb{H}^2\times \{h\}$ and outer radius $r+d$ inside $\beta\subset \mathbb{H}^2\times \{h\}$ has mean curvature footnotesizeer than $2H.$ Besides, if $\coth(r) > 2H,$ condition \eqref{eq-hipsubs} is equivalent to $h\le \cat_r(d)$ for footnotesize $d$ and $h\le \frac{2Hd}{\sqrt{\coth^2(r)-4H^2}}$ for large $d.$ The change of behavior occurs at $d_0$ the only solution of
$$\cat_r(d)= \frac{2Hd}{\sqrt{\coth^2(r)-4H^2}}.$$ It is natural to expect that for footnotesize values of $d,$ the catenoid gives a higher subsolution because it has initial vertical slope. On the other hand, for large $d,$ since the catenoid is asymptotic to plane, the cone is a higher subsolution.


\end{enumerate}


\section{Proofs of the non existence results}

In this section, we will prove the non existence result about a constant mean curvature surface contained in a slab of $\mathbb{H}^2 \times \mathbb{R}$ using a condition involving the distance between the slices that bound the slab and the given constant $H.$ The proof was inspired by Proposition 3.2 of \cite{Jaime}. Other similar developments concerning this problem were recently proposed by \cite{AT} to the case $H = 0$ and  by \cite{Senni} to  the case $H \in (0, 1/2).$

\

\begin{proof}[Proof of Theorem \ref{theo-nonexist}]


Let $M$ be as in Theorem \ref{theo-nonexist} and let $B$ be an horodisk such that $M\subset B\times \mathbb{R}$, and denote by $C=\partial B$. For each $t>0$, let $C_t\subset B$ be the horocycle at distance $t$ of $C$. Consider $\mathcal{S}_{t}$ the vertical translation by $h_H/2$ upwards of the $\Hhorocoisa$ of neck $C_t.$


It follows immediately from the fact that $\{\mathcal{S}_t\}_{t\ge 0}$ foliates $B\times[0,h_H]$ that there exists $T>0$ such that $M$ and $\mathcal{C}_{T}$ have a tangency point. If we assume that $h> h_H$, this point cannot be on $\partial M$ and therefore it must occur an interior tangency point, which implies, by the Tangency Principle, that $M$ coincides with $\mathcal{C}_T$, an absurdity. Hence $h\le h_H$, concluding the proof.
\end{proof}

\

\begin{proof}[Proof of Theorem \ref{teo-nonexistBr}]
We argue as in the proof of Theorem \ref{theo-nonexist}. 

For $r>0,$ let $\mathcal{H}_r$ be the surface obtained rotating the graph of $\Hnod_r(s)$ in  restricted to $[0,x_H(r)]$ glued with its reflection around the plane $\HH\times \{0\},$ i. e., $$\mathcal{H}_r=\{\left(p, \pm\Hnod_r(s(p))\right)\, \vert \, 0\leq s(p)\leq  x_H(r)\},$$
for $s(p)$ the distance between $p$ and $\partial B_r.$
For $r\in (0,r^*),$ a member of the family $\mathcal{H}_r$
must have an interior contact point with $M$ if $h> 2\max \Hnod_{r^*}.$
\end{proof}

%
%

\centerline{
\begin{tabular}{ccccc}
{\footnotesize Rodrigo B. Soares} & $\hspace{0.1cm}$&{\footnotesize Patrícia K. Klaser} & $\hspace{0.1cm}$& {\footnotesize Miriam Telichevesky} \\
{\footnotesize Univ. Federal do Rio Grande}  &  &{\footnotesize Univ. Federal do Rio Grande do Sul}  &  & {\footnotesize Univ. Federal do Rio Grande do Sul}\\
{\footnotesize Inst. de Matem\'atica, Estat\'istica e F\'isica} &
&{\footnotesize Instituto de Matem\'atica} &
& {\footnotesize Instituto de Matem\'atica}\\
{\footnotesize Av. It\'alia, Km 8} & & {\footnotesize Av. Bento Gon\c calves 9500} & & {\footnotesize Av. Bento Gon\c calves 9500}\\
{\footnotesize 96201-900 Rio Grande-RS } & &{\footnotesize 91509-500 Porto Alegre-RS } & &{\footnotesize 91509-500 Porto Alegre-RS }\\
{\footnotesize  BRASIL} &  &{\footnotesize  BRASIL} &  & {\footnotesize BRASIL} \\
{\footnotesize rodrigosoares@furg.br}& &{\footnotesize patricia.klaser@ufrgs.br}& &{\footnotesize miriam.telichevesky@ufrgs.br} \\
\end{tabular}}

\end{document}